\documentclass{amsart}
\usepackage{pst-plot}
\usepackage{graphicx}
\usepackage{epstopdf}
\usepackage[english]{babel}
\usepackage[latin1]{inputenc}
 \usepackage[all]{xy}

\usepackage{amsmath,amsfonts,amssymb,amsthm,amscd,array,stmaryrd,mathrsfs, mathdots}
\PassOptionsToPackage{option}{xcolor}
\usepackage{pstricks}

\theoremstyle{plain}

\newtheorem{lem}{Lemma}
\newtheorem{thm}[lem]{Theorem}

\newtheorem{conj}[lem]{Conjecture}

\theoremstyle{definition}
\newtheorem*{rem}{Remark}
\newtheorem{ex}{Example}
\newtheorem{exe}{Exercise}

\newcommand{\R}{\mathbb{R}}
\newcommand{\Z}{\mathbb{Z}}
\newcommand{\C}{\mathbb{C}}
\newcommand{\N}{\mathbb{N}}

\newcommand{\Q}{\mathbb{Q}}

\newcommand{\F}{\mathcal{F}}

\newcommand{\cN}{\mathcal{N}}
\newcommand{\cM}{\mathcal{M}}
\newcommand{\cR}{\mathcal{R}}
\newcommand{\Qc}{\mathcal{Q}} 
\newcommand{\Pc}{\mathcal{P}}
\newcommand{\Sc}{\mathcal{S}}

\newcommand{\Id}{\mathrm{Id}}
\newcommand{\SL}{\mathrm{SL}}
\newcommand{\PSL}{\mathrm{PSL}}
\newcommand{\PGL}{\mathrm{PGL}}

\def\a{\alpha}

\def\D{\Delta}

\begin{document}

\title[Quantum real numbers]
{$q$-deformed rationals and irrationals}

\author[S. Morier-Genoud]{Sophie Morier-Genoud}
\author[V. Ovsienko]{Valentin Ovsienko}

\address{Sophie Morier-Genoud,
Laboratoire de Math\'ematiques,
Universit\'e de Reims,
U.F.R. Sciences Exactes et Naturelles,
Moulin de la Housse - BP 1039,
51687 Reims cedex 2,
France
}

\address{
Valentin Ovsienko,
Centre National de la Recherche Scientifique,
Laboratoire de Math\'ematiques,
Universit\'e de Reims,
U.F.R. Sciences Exactes et Naturelles,
Moulin de la Housse - BP 1039,
51687 Reims cedex 2,
France}

\email{sophie.morier-genoud@univ-reims.fr,
valentin.ovsienko@univ-reims.fr}

\maketitle

The concept of $q$-deformation, or ``$q$-analogue'' arises in many areas of mathematics. 
In algebra and representation theory, it is the origin of quantum groups; 
$q$-deformations are important for knot invariants, combinatorial enumeration, discrete geometry, analysis, and many other parts of mathematics.
In mathematical physics, $q$-deformations are often understood as ``quantizations''.

The recently introduced notion of a $q$-deformed real number is based on the geometric idea of invariance by a modular group action. 
The goal of this lecture is to explain what is a $q$-rational and a $q$-irrational, 
demonstrate beautiful properties of these objects, and describe their relations to many different areas.
We also tried to describe some applications of $q$-numbers.

The literature on the subject grows fast, we were not able to give a detailed account.
Some topics were left outside the scope of this lecture, such as relation to the Jones polynomial,
relation to the Burau representation of Braid groups,
as well as topological and homological aspects of the theory.
We were only able to give complete proofs for some of the theorems, 
but we tried to outline the ideas of the proofs of most of them.
We tried to explain the main ideas and perspectives of this recent theory, 
some surprises are waiting for the reader in the end.

\section{$q$-analogues: counting better!}

Before we start, we need to explain the role of $q$-deformations in combinatorics.
In one sentence, one can say the following:
{\it whenever a combinatorial objects counts something,
its $q$-analogue counts exactly the same thing, but more precisely}.
This is a very large and classical theory, our explanation will be based on some examples.

\subsection{Main example: Gaussian $q$-binomials}
The polynomial
\begin{equation}
\label{qBn}
[n]_{q}:=1+q+q^{2}+\cdots+q^{n-1}=\textstyle\frac{1-q^n}{1-q},
\end{equation}
where $q$ is a parameter,
 is commonly considered as a ``quantum'', or a $q$-analogue of a (positive) integer~$n$. 
 The expression goes back to Euler ($\approx$1760) who used it in the context of combinatorics and ``$q$-series''.
 Lecture~3 of this book explains the remarkable Euler function $\varphi(q)$ where the
 polynomials~$[n]_{q}$ appear naturally.
 
 Consequently, Gauss ($\approx$1808) introduced polynomials based on $[n]_{q}$, called the
 $q$-binomials.
 Their formula is similar to the formula for the usual binomial coefficients
\begin{equation}
\label{qBin}
{n\choose m}_q=
\frac{[n]_q!}{[n-m]_q!\,[m]_q!},
\end{equation}
where $[n]_q!$ is the $q$-factorial:
\begin{equation}
\label{qfac}
[n]_q!=[1]_q[2]_q\cdots[n]_q.
\end{equation}
The latter polynomial is also very interesting and is related to many topics of algebra and combinatorics.
Some of them were discussed in Lecture~3 of this book. 

Gauss himself introduced the polynomials~\eqref{qBin} to solve some problems
of number theory related to so-called ``quadratic Gauss sums''.
Gaussian $q$-binomial coefficients have been used in many dirrefent branches of mathematics
illustrating the general principle.
{\it There are more theorems than formulas! 
The same formula describes many different phenomenons.}

The fact that the $q$-binomials~\eqref{qBin} are, indeed, polynomials in~$q$, and not rational functions,
is not completely obvious. 
This can be proved by induction using the following $q$-analogue of the Pascal identity.
\begin{exe}
\label{PascalEx}
Prove that
\begin{equation}
\label{WPI}
{n\choose m}_q={n-1\choose m-1}_q+q^m{n-1\choose m}_q.
\end{equation}
Calculate the first interesting example ${4\choose 2}_q=1+q+2q^2+q^3+q^4$.
\end{exe}

The $q$-binomial coefficients are the vertices of the ``weighted Pascal triangle'' which encodes the above formula.
$$
\xymatrix @!0 @R=0.6cm @C=0.8cm
{
&&&&{0\choose 0}_q\ar@{-}[rdd]^{\textcolor{red}{1}}\ar@{-}[ldd]_{\textcolor{red}{1}}
\\
\\
&&&{1\choose 0}_q\ar@{-}[rdd]^{\textcolor{red}{1}}\ar@{-}[ldd]_{\textcolor{red}{1}}
&&{1\choose 1}_q\ar@{-}[rdd]^{\textcolor{red}{1}}\ar@{-}[ldd]_{\textcolor{red}{q}}
\\
\\
&&{2\choose 0}_q\ar@{-}[rdd]^{\textcolor{red}{1}}\ar@{-}[ldd]_{\textcolor{red}{1}}
&&{2\choose 1}_q\ar@{-}[rdd]^{\textcolor{red}{1}}\ar@{-}[ldd]_{\textcolor{red}{q}}&&{2\choose 2}_q\ar@{-}[rdd]^{\textcolor{red}{1}}\ar@{-}[ldd]_{\textcolor{red}{q^2}}
\\
\\
&{3\choose 0}_q\ar@{-}[rdd]^{\textcolor{red}{1}}\ar@{-}[ldd]_{\textcolor{red}{1}}
&&{3\choose 1}_q\ar@{-}[rdd]^{\textcolor{red}{1}}\ar@{-}[ldd]_{\textcolor{red}{q}}
&&{3\choose 2}_q\ar@{-}[rdd]^{\textcolor{red}{1}}\ar@{-}[ldd]_{\textcolor{red}{q^2}}
&&{3\choose 3}_q\ar@{-}[rdd]^{\textcolor{red}{1}}\ar@{-}[ldd]_{\textcolor{red}{q^3}}
\\
\\
{4\choose 0}_q
&&{4\choose 1}_q&&{4\choose 2}_q&&{4\choose 3}_q&&{4\choose 4}_q
\\ 
&&&\cdots&&\cdots
}
$$

Some properties of the polynomials~\eqref{qBin}  are not very difficult to prove.
They have positive integer coefficients and they are {\it monic}, i.e.
the leading coefficient is equal to~$1$.
These polynomials are {\it palindromic} (or ``self-reciprocal''), that is their
coefficients form a symmetric sequence.
A more difficult and important property of the polynomials~\eqref{qBin} is described by the following classical theorem
 of Sylvester (1878, \cite{Syl}), who solved
Cayley's conjecture formulated 25 years earlier.

\begin{thm}
\label{SylThm}
The polynomials ${n\choose m}_q$ are unimodal.
\end{thm}
This means that the sequence of coefficients form a 
``lonely mountain'', with no oscillation: the coefficients (weakly) monotonously increase 
from $1$ to some maximum and then monotonously decrease back to~$1$.
This is an important property because unimodal sequences frequently appear in geometry and combinatorics; see~\cite{Sta}.
Proud of his theorem Sylvester wrote:
``I accomplished with scarcely an effort a task which I had believed lay outside the rang of human power.''
Several modern proofs of this theorem are known; see, e.g., \cite{Zei}, but none of them is elementary.

\subsection{What is... a $q$-analogue?}
A meaningful $q$-analogue must satisfy several requirements.

{\bf A}. 
Given a sequence of integers that counts some objects,
its $q$-analogue must count the same
objects, but with more details.
To illustrate this general principle, we will give only one example.

The classical  binomial coefficient
${n\choose m}$ counts the number of {\it north-east lattice paths}, or ``NS-paths''
(with no steps down or left) in  the $n\times(n-m)$-rectangle.
To simplify the language, we call $1\times1$ squares in the grid ``boxes''.
A classical theorem states the following.

\begin{thm}
\label{NEThm}
The coefficient of~$q^k$ in the polynomial~${n\choose m}_q$ is the number of NE-paths with
exactly~$k$ boxes under the path.
\end{thm}

\begin{ex}
The polynomial ${4\choose 2}_q=1+q+2q^2+q^3+q^4$, counts  NE-paths in a 
$2\times2$-square.
There are ${4\choose 2}=6$ such paths. 
Furthermore, there is exactly one such path with either $0,1,3$, or $4$ under boxes,
and there are two paths with $2$ under boxes:
$$
\xymatrix @!0 @R=0.6cm @C=0.7cm
{
\bullet\ar@{-}[r]\ar@{-}[dd]&{\color{red}\bullet}\ar@{<=}[d]\ar@{=>}[r]
&{\color{red}\bullet}\ar@{-}[dd]\\
\bullet\ar@{-}[rr]&{\color{red}\bullet}\ar@{<=}[d]&\bullet
\\
{\color{red}\bullet}\ar@{=>}[r]
&{\color{red}\bullet}\ar@{-}[r]&\bullet
}
\qquad\qquad
\xymatrix @!0 @R=0.6cm @C=0.7cm
{
\bullet\ar@{-}[rr]\ar@{-}[d]&\bullet\ar@{-}[dd]
&{\color{red}\bullet}\ar@{<=}[d]\\
{\color{red}\bullet}\ar@{=>}[r]\ar@{<=}[d]&{\color{red}\bullet}\ar@{=>}[r]&{\color{red}\bullet}\ar@{-}[d]
\\
{\color{red}\bullet}\ar@{-}[rr]
&\bullet&\bullet
}
$$
This corresponds to the coefficients $1,1,2,1,1$ in the polynomial ${4\choose 2}_q$.
\end{ex}

Theorem~\ref{NEThm} means that {\it every coefficient} of the polynomial~${n\choose m}_q$
has a combinatorial meaning.
Every coefficient counts NE-paths of the same type (with the same number of under boxes).

\begin{rem}
A NE-path is also called a Young diagram.
This important notion is used in representation theory and combinatorics.
\end{rem}

{\bf B}.
Another requirement for a meaningful $q$-analogue is geometric.
Let us quote the classical book~\cite{StaEnum}.
{\it To be a ``satisfactory'' $q$-analogue more is required, but there is no precise definition of what is meant by ``satisfactory''.
Certainly one desirable property is that the original object concerns finite sets, 
while the q-analogue can be interpreted in terms of subspaces of 
finite-dimensional vector spaces over the finite field~$\mathbb{F}_q$.}
We will just touch on this and never use.
The reader unfamiliar with such higher algebra terms as ``Grassmannians'', or ``finite fields'' can skip this part.

A $q$-analogue must
count the number of points in a certain algebraic variety over the finite field~$\mathbb{F}_q$,
where $q=p^\ell$, and where~$p$ is a prime integer.
For instance, for our main characher, the Gaussian binomial we have the following.

\begin{thm}
The polynomial~${n\choose m}_q$, is equal to the number of 
$m$-dimensional subspaces in the $n$-dimensional space~$\mathbb{F}_q^n$.
\end{thm}

A simple proof can be found in Stanley's book~\cite{StaEnum}.
Once again, every coefficient has it's own meaning, it counts the number of points
in the Schubert celles of the same dimension!

{\bf C}. 
A more analytic requirement:
the same $q$-analogue
appears in ``quantum calculus'', or ``quantum algebra''.
Once again, we give an example related to the $q$-binomials.
Assume that the non-commuting variables $x$ and $y$ satisfy the relation $yx=qxy$ of the ``quantum plane'';
the $q$-binomial theorem then states that
$$
(x+y)^n=\sum_{0\leq{}m\leq{}n}{n\choose m}_qx^my^{n-m},
$$
producing the $q$-binomials.

The reason for which the same polynomial appear in completely different situations is a great mystery!

\begin{rem}
 A historical anecdote can be told here.
By tradition, the letter $q$ is often used to denote the variable in~\eqref{qBn}
 (and not $x,y,z,\ldots$).
 Euler used $q$ as an abbreviation for ``quotient'', and hundreds of authors did so after him
 well before the quantum era.
 The same letter $q$ is used in the theory of quantum groups, quantum calculus, etc., where the polynomials~\eqref{qBn}
 appear inevitably, but as an abbreviation for ``quantum''.
 Speculation, or miraculous coincidence, who knows?
 
 The skeptical reader may ask, but why the same letter $q$ denotes the finite field~$\mathbb{F}_q$?
 Well, in this context, $q$ is a power of a prime number: $q=p^k$ and $p$ stays for ``prime'', while
 $q$ is alphabetically the next letter!
 Not very convincing but you can't escape fate!
 \end{rem}

\subsection{What is... the $q$-analogue of a rational number?}
Let $x=\frac{n}{m}$ be a rational number, we want to define its $q$-analogue.
A very naive idea would be to consider the quotient of two $q$-integers $\frac{[n]_q}{[m]_q}$.
However, such a quotient contain no additional information, and does not satisfy any version of
properties {\bf A}, {\bf B}, and {\bf C}.
The expression $\frac{1-q^x}{1-q}$ is another tempting candidate, but replacing $q$ by~$q^m$
we obtain exactly the same non-interesting quotient $\frac{[n]_q}{[m]_q}$.

The $q$-analogue discussed here is based on
the geometric idea of invariance by a group action.
The set of rationals completed by one additional point, $\Q\cup\{\infty\}$, admits a transitive action of
the group $\PSL(2,\Z)$ of unimodular $2\times2$ matrices with integer coefficients, called the {\it modular group}.
This group also naturally act on the space $\Z(q)$ of rational functions in~$q$ with integer coefficients.
We will describe a map 
$$
[\,.\,]_q:\Q\cup\{\infty\}\to\Z(q)
$$ 
commuting with the $\PSL(2,\Z)$-action.
Such a map is unique if we know the image of one point, say if we know that $[0]_q=0$.
We thus define a $q$-analogue of a rational number which is an irreducible quotient of
two polynomials with integer coefficients: $\left[\frac{n}{m}\right]_{q}=\frac{\cN(q)}{\cM(q)}$.
Both polynomials, $\cN(q)$ and~$\cM(q)$, depend on $n$ and~$m$, and for $q=1$ we recover the initial rational:
$\cN(1)=n$ and $\cM(1)=m$.

The notion of $q$-analogue of rationals defined this way satisfies many interesting properties,
some of them are very similar to those of the $q$-binomials.
The main property is the ``stabilization phenomenon'' consisting in convergence of power series.
It leads to the notion of
$q$-analogue of an irrational number.

\section{The group of symmetries}\label{SymSec}

Rational numbers will always be represented by reduced quotients $\frac{n}{m}$ of two integers.
We will be considering the set $\Q\cup\{\infty\}$ of rationals, completed by one additional point, $\infty$,
represented by~$\frac{1}{0}$.

\subsection{The action of $\SL(2,\Z)$ on rationals, the modular group $\PSL(2,\Z)$}
The set $\Q\cup\{\infty\}$ admits an action of
the group $\SL(2,\Z)$ by fractional-linear transformations
\begin{equation}
\label{LFAct}
\begin{pmatrix}
a&b\\
c&d
\end{pmatrix}\left(x\right)=
\frac{ax+b}{cx+d},
\qquad\qquad
a,b,c,d\in\Z,
\quad
ad-bc=1.
\end{equation}

\begin{exe}
\label{TransEx}
(i) Check that~\eqref{LFAct} is indeed an action, i.e. if $A,B\in\SL(2,\Z)$
and $x\in\Q\cup\{\infty\}$, then $AB(x)=A(B(x))$.

(ii) Prove that the action~\eqref{LFAct} is {\it transitive}, i.e., for every~$x,y\in\Q\cup\{\infty\}$,
there exists $A\in\SL(2,\Z)$, such that $A(x)=y$.
\end{exe}

The set $\Q\cup\{\infty\}$ equipped with this action is naturally understood as the rational projective line
$\Q\mathbb P^1$.
The center of $\SL(2,\Z)$ consists of two matrices, $\pm\Id$, and its acts trivially.
The action becomes efficient, or faithful, if one considers the quotient  
$$
\PSL(2,\Z):=\SL(2,\Z)/\{\pm\Id\}
$$
called the modular group.

Let us describe the structure of $\PSL(2,\Z)$ in an ``algebraic manner'',
that is in terms of generators and relations.
Such description of a group is usually called a {\it presentation}.
We have the following old classical theorem.

\begin{thm}
\label{GenRel}
(i) The group $\PSL(2,\Z)$ is generated by two elements
$$
T=
\begin{pmatrix}
1&1\\[2pt]
0&1
\end{pmatrix},
\qquad\qquad
S=
\begin{pmatrix}
0&-1\\[2pt]
1&0
\end{pmatrix}.
$$

(ii) These generators satisfy the relations
\begin{equation}
\label{Rel}
S^2=(TS)^3=1.
\end{equation}

(iii)$^*$ Every relation between $S$ and $T$ is a consequence of~\eqref{Rel}.
\end{thm}

An elementary proof can be found in~\cite{Alp}.
The reader can deduce Part~(i) from the Euclid algorithm, and then easily check~\eqref{Rel}.
Part~(iii) is more involved.

\begin{rem}
Note that~$1$ in the right hand side of~\eqref{Rel} is due to the quotient of $\SL(2,\Z)$ by the center.
The square of $S$ and the cube of $TS$ are equal to $-\Id$, the negation of the identity matrix, 
but the corresponding linear-fractional transformations
$$
T(x)=
x+1,
\qquad\qquad
S(x)=-\frac{1}{x}
$$
satisfy~\eqref{Rel}.
\end{rem}

\subsection{The action of $\PSL(2,\Z)$ on rational functions}\label{PSLAct}
Let us describe a natural action of $\PSL(2,\Z)$ on the space of rational functions $\Z(q)$.

Consider the matrices
$$
T_{q}=
\begin{pmatrix}
q&1\\[2pt]
0&1
\end{pmatrix},
\qquad\qquad
S_{q}=
\begin{pmatrix}
0&-1\\[2pt]
q&{\phantom -}0
\end{pmatrix}
$$
depending on a formal variable~$q$.
Let $T_{q}$ and $T_{q}$ act on~$\Z(q)$ by linear-fractional transformations~\eqref{LFAct}:
\begin{equation}
\label{AcTq}
T_q\left(X(q)\right)=qX(q)+1,
\qquad\qquad
S_q\left(X(q)\right)=-\frac{1}{qX(q)}.
\end{equation}

\begin{exe}
\label{SameRel}
Check that the maps~\eqref{AcTq}, 
that (slightly abusing the notation) we also denote by $T_{q}$ and $S_{q}$, 
satisfy exactly the same relations as $T$ and $S$:
\begin{equation}
\label{Relq}
S_q^2=(T_qS_q)^3=1.
\end{equation}
This implies that the operators~\eqref{AcTq} generate a $\PSL(2,\Z)$-action on $\Z(q)$.
\end{exe}

\begin{rem}
Let us explain the reason for which we have chosen the $\PSL(2,\Z)$-action generated by~\eqref{AcTq}.
The polynomials~\eqref{qBn} satisfy the obvious relation
$[n+1]_q=q[n]_q+1.$
The operator $T_q$ in~\eqref{AcTq} is chosen to match with the
$q$-integers.
The operator~$S_{q}$ turns out to be the only operator satisfying the relations~\eqref{Relq} with the chosen~$T_q$.
\end{rem}

\section{Introducing $q$-rationals}\label{ItroSec}

Whenever we have an action of some group on two different sets, it is natural to study invariant maps.
The notion of $q$-deformed rational numbers is entirely based on this idea.

\begin{thm}
\label{ExThm}
There exists a unique map
$[\,.\,]:\Q\cup\{\infty\}\to\Z(q)$
commuting with the $\PSL_2(\Z)$-action and such that $[0]_q=0.$
\end{thm}

The goal of this section is to describe this quantization map explicitly and give  examples of $q$-rationals.
We also outline the proof of the theorem.
The uniqueness statement follows from transitivity of the $\PSL_2(\Z)$-action on~$\Q\cup\{\infty\}$,
the existence will be established by providing explicit formulas.
Note that although zero seems to be the most natural choice of ``quantum zero'', but there is a second one;
see Section~\ref{FlatSec} below.

\subsection{Continued fractions}
To give an explicit formula of $q$-rationals, we will use continued fractions, 
a classical way to represent numbers; see Lecture~1.

Every rational number~$\frac{n}{m}$ can be written as a finite continued fraction
$$
\frac{n}{m}
\quad=\quad
a_0 + \cfrac{1}{a_1 
          + \cfrac{1}{\ddots +\cfrac{1}{a_{k}} } },
          $$
with integer coefficients~$a_i$, such that $a_i\geq1$ for $i\geq1$ and arbitrary~$a_0$.
It is usually called the regular continued fraction expansion and is denoted by  $\frac{n}{m}=[a_0,a_1,\ldots,a_{k}]$.
The continued fraction expansion is (almost!) unique,
except for the ambiguity $[a_0,\ldots,a_{k},1]=[a_0,\ldots,a_{k}+1]$
that can be removed by choosing an even or odd number of coefficients.

There is also a unique continued fraction expansion 
with minus signs
$$
\frac{n}{m}
 \quad =\quad
c_0 - \cfrac{1}{c_1 
          - \cfrac{1}{\ddots - \cfrac{1}{c_\ell} } } ,
$$
where $c_j$ are integers such that $c_j\geq2$ (except for $c_0$).
It is less known and is someztimes called the ``Hirzebruch continued fraction''.
The notation for such expansion used by Hirzebruch is $\frac{n}{m}=\llbracket{}c_0,c_1,\ldots,c_\ell\rrbracket{}$.
The coefficients $a_i$ and $c_j$ of the two expansions
are connected by a Hirzebruch formula~\cite{Hir}.
It is interesting to mention that Coxeter also used continued fractions with minus signs
connecting them with his frieze patterns; see Lecture 5.

\subsection{Exact formulas}
The explicit formula for $q$-rationals is easier to obtain in terms of the Hirzebruch expansion.
Every coefficient $c_j$ is replaced by its $q$-analogue~\eqref{qBn}, and every~$1$ in the numerators by a power of~$q$.
More precisely, we have the following.

\begin{lem}
\label{FirstLm}
The map $\frac{n}{m}\mapsto\left[\frac{n}{m}\right]_q$, where
\begin{equation}
\label{qc}
\left[\frac{n}{m}\right]_q
 \quad =\quad
[c_0]_{q} - \cfrac{q^{c_{0}-1}}{[c_1]_{q} 
          - \cfrac{q^{c_{1}-1}}{\ddots \cfrac{\ddots}{[c_{\ell-1}]_{q}- \cfrac{q^{c_{\ell-1}-1}}{[c_\ell]_{q}} } }} 
\end{equation}
commutes with the $\PSL_2(\Z)$-action.
\end{lem}

\begin{proof}
We will proceed by induction on the number of coefficients~$\ell$.
For integers~\eqref{qc} coincides with~\eqref{qBn}.
Assume that~\eqref{qc} is true for $\frac{n'}{m'}=\llbracket{}c_1,\ldots,c_\ell\rrbracket{}$.
Then $\frac{n}{m}=c_0-\frac{m'}{n'}$.
In other words, $\frac{n}{m}$ is obtained from~$\frac{n'}{m'}$ applying the generators 
 $S$ and $T$ of $\PSL_2(\Z)$ as follows
$$
\frac{n}{m}=T^{c_0}S\left(\frac{n'}{m'}\right).
$$
Given $\left[\frac{n'}{m'}\right]_q$, the $\PSL_2(\Z)$-invariance then reads
$
\left[\frac{n}{m}\right]_q=
T_q^{c_0}S_q\left(\left[\frac{n'}{m'}\right]_q\right).
$
Using the iteration of~\eqref{AcTq}, one then obtains 
$$
\left[\frac{n}{m}\right]_q=
\left[c_0\right]_q-q^{c_0-1}\left[\frac{m'}{n'}\right]_q.
$$
and hence~\eqref{qc}.
\end{proof}

The following recurrence is a corollary of~\eqref{qc}.
Let $\frac{n_i}{m_i}=\llbracket{}c_0,\ldots,c_i\rrbracket{}$ be the consecutive convergents of
the Hirzebruch continued fraction,
and $\left[\frac{n_i}{m_i}\right]_q=:\frac{\cN_i(q)}{\cM_i(q)}$ be their quantization, then
\begin{equation}
\label{Recur}
\begin{array}{rcl}
\cN_{i+1} &=& [c_{i+1}]_q\cN_{i}-q^{c_i-1}\cN_{i-1}, \\[4pt]
\cM_{i+1} &=& [c_{i+1}]_q\cM_{i}-q^{c_i-1}\cM_{i-1},
\end{array}
\end{equation}

The explicit formula in terms of the regular continued fraction expansion is more complicated.
The coefficients $a_i$ are again replaced by their $q$-analogues according to~\eqref{qBn}, but this time
the parameter~$q$ is changed to~$q^{-1}$ when~$i$ is odd.

\begin{exe}
\label{Ex9}
Check that~\eqref{qa} commutes with the $\PSL_2(\Z)$-action and thus is consistent with~\eqref{qc}.
\begin{equation}
\label{qa}
\left[\frac{n}{m}\right]_q
 \quad =\quad
[a_0]_{q} + \cfrac{q^{a_{0}}}{[a_1]_{q^{-1}} 
          + \cfrac{q^{-a_{1}}}{[a_{2}]_{q} 
          +\cfrac{q^{a_{2}}}{[a_{3}]_{q^{-1}}
          + \cfrac{q^{-a_{3}}}{
        \cfrac{\ddots}{+\cfrac{q^{(-1)^{k-1}a_{k-1}}}{[a_{k}]_{q^{(-1)^k}}}}}
          } }} 
\end{equation}
\end{exe}

Formulas~\eqref{qc} and~\eqref{qa} imply the existence part of Theorem~\ref{ExThm}.

\subsection{Examples of $q$-rationals, $q$-Fibonacci and $q$-Pell numbers}\label{FiboPell}
We start with simple examples of $q$-rationals.

\begin{exe}
\label{LEx}
(i)
Show that exactly three elements of $\Q\cup\infty$ are preserved by the quantization map
$$
[0]_q=0,
\qquad 
[1]_q=1, 
\qquad 
\left[\frac{1}{0}\right]_q=\frac{1}{0}.
$$
For any other $x\in\Q\cup\infty$ the result $\left[x\right]_q$ depends on~$q$ explicitly.

(ii)
Calculate using~\eqref{qc} or~\eqref{qa}
$$
\left[\frac{1}{2}\right]_{q}=
\frac{q}{1+q},
\qquad
\left[\frac{5}{2}\right]_{q}=
\frac{1+2q+q^{2}+q^{3}}{1+q},
\qquad
\left[\frac{5}{3}\right]_{q}=
\frac{1+q+2q^{2}+q^{3}}{1+q+q^{2}}.
$$
\end{exe}

The last example is a term of a quite remarkable series of rationals $\frac{F_{n+1}}{F_n}$,
where $F_n=1,1,2,3,5,8,13,\ldots$ are  the  Fibonacci numbers.
The consecutive quotients, $\frac{F_{n+1}}{F_n}$, lead to an interesting sequence of rational functions:
$[1]_q=1,\,[2]_q=1+q,\,\left[\frac{3}{2}\right]_q=\frac{1+q+q^2}{1+q},$ and then $\left[\frac{5}{3}\right]_{q}$,
followed by
$$
\begin{array}{rcl}
\left[\frac{8}{5}\right]_q&=&\displaystyle\frac{1+2q+2q^2+2q^3+q^4}{1+2q+q^2+q^3},
\\[12pt]
\left[\frac{13}{8}\right]_q&=&\displaystyle\frac{1+2q+3q^2+3q^3+3q^4+q^5}{1+2q+2q^2+2q^3+q^4},
\\[12pt]
\left[\frac{21}{13}\right]_q&=&\displaystyle\frac{1+3q+4q^2+5q^3+4q^4+3q^5+q^6}{1+3q+3q^2+3q^3+2q^4+q^5},
\\
\ldots&\ldots&\ldots
\end{array}
$$
We obtain
$\left[\frac{F_{n+1}}{F_{n}}\right]_q=
\frac{\tilde\F_{n+1}(q)}{\F_n(q)},$ 
the polynomials in the numerator and denominator are of degree~$n-2$ (for~$n\geq2$)
and are mirror of each other:
$
\tilde\F_n(q)=q^{n-2}\F_n(q^{-1}).
$
They can be calculated using the recurrence
$$
\F_{n+2}(q)= [3]_q\,\F_n(q)-q^2\F_{n-2}(q),
$$
which follows from~\eqref{Recur}.
 
The coefficients of the polynomials in the numerators and
denominators turn out to match the known triangular integer sequences~A123245 and~A079487 of~OEIS,
(which are mirrors of each other).

\begin{rem}
Every time a new procedure leads to something known, it is a good news for a researcher
working on the subject, since known properties can shade light on the whole picture.
Such properties can eventually be generalized and become general theorems.
This happened with $q$-Fibonacci numbers and the golden ratio, see Section~\ref{GoldSec} below
and was crucial for \cite{MGOfmsigma} and~\cite{MGOexp}.
\end{rem}

The $q$-rational $\left[\frac{5}{2}\right]_{q}$ of Example~\ref{LEx} is also a starting point of an interesting series.
Consider the classical sequence of Pell numbers $P_n=1, 2, 5, 12, 29, 70, 169,\ldots$
The Pell numbers are almost as famous and old as the Fibonacci numbers.
They are characterized by two first values $1,2$ and the recurrence 
$$
P_n=2P_{n-1}+P_{n-2}.
$$
The rational sequence~$\frac{P_{n+1}}{P_n}$ is known (since Pythagorean time) to approximate $1+\sqrt{2}$.
The $q$-analogue of the rational sequence~$\frac{P_{n+1}}{P_n}$ starts as follows.
$$
\begin{array}{rcl}
\left[\frac{12}{5}\right]_q&=&\displaystyle\frac{1+2q+3q^2+3q^3+2q^4+q^5}{1+q+2q^2+q^3},
\\[12pt]
\left[\frac{29}{12}\right]_q&=&\displaystyle\frac{1+3q+5q^2+6q^3+6q^4+5q^5+2q^6+q^7}{1+2q+3q^2+3q^3+2q^4+q^5},
\\[12pt]
\left[\frac{70}{29}\right]_q&=&\displaystyle\frac{1+3q+7q^2+11q^3+13q^4+13q^5+11q^6+7q^7+3q^8+q^9}{1+2q+5q^2+6q^3+6q^4+5q^5+3q^6+q^7},
\\
\ldots&\ldots&\ldots
\end{array}
$$
It produces a $q$-analogue of the sequence of Pell numbers, that
was not known before~\cite{MGOfmsigma} and was recently included into OEIS as sequence A323670.

\begin{rem}
The above examples suggest a general procedure that can be applied to quantize any integer sequence.
Given a sequence $(a_n)_{n\in\N}$, and take the sequence of rationals $\frac{a_{n+1}}{a_n}$
and $q$-rationals $\left[\frac{a_{n+1}}{a_n}\right]_q$.
The polynomials arising in this procedure can be considered as a $q$-analogue of the sequence~$(a_n)$.
Many surprises await to be discovered on this road!
\end{rem}

\section{The ``weighted'' Farey graph}

We describe a simple inductive way to calculate $q$-rationals.
It is very similar to the process of calculating $q$-binomials from the weighted Pascal triangle,
with the only difference that the Pascal triangle is replaced by
the {\it Farey graph}, already considered in Lecture~5.

All rationals, completed by~$\infty$ represented by~$\frac{1}{0}$, are vertices of the Farey graph.
The set of edges is constructed as follows.
Two rationals,
$\frac{n}{m}$ and~$\frac{r}{s}$ are connected by an edge if and only if $|ns-mt|=1$.

\begin{center}
\includegraphics[width=0.7\textwidth]
{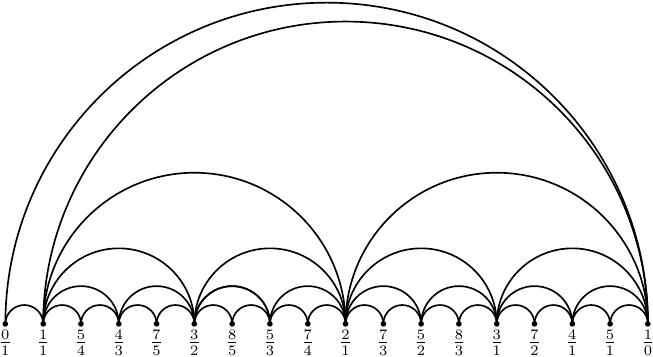}

A part of the Farey graph.
\end{center}
Every rational number belongs to infinitely many ``elementary triangles'' of the form
\begin{center}
\includegraphics[width=0.2\textwidth]
{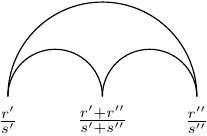}
\end{center}
The rational~$\frac{r'+r''}{s'+s''}$ can be considered as ``descendent'' of $\frac{r'}{s'}$ and~$\frac{r''}{s''}$,
it is called the mediant, or the Farey sum of $\frac{r'}{s'}$ and~$\frac{r''}{s''}$.
Starting from the ``initial triangle'' $\frac{0}{1},\frac{1}{1},\frac{1}{0}$,
and calculating the consecutive mediants, we can obtain every positive rational number.

It turns out that the $q$-rationals can be calculated inductively using the $q$-deformed triangles
\begin{center}
\includegraphics[width=0.2\textwidth]
{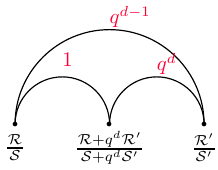}
\end{center}
the $q$-deformed mediant of two rational functions, $\frac{\cN(q)}{\cM(q)}$ and~$\frac{\cR(q)}{\Sc(q)}$ being given by
\begin{equation}
\label{WFS}
\frac{\cN(q)}{\cM(q)}\oplus\frac{\cR(q)}{\Sc(q)}=
\frac{\cN(q)+q^d\cR(q)}{\cM(q)+q^d\Sc(q)},
\end{equation}
where $d$ is calculated inductively, staring with
the initial triangle
\begin{center}
\includegraphics[width=0.2\textwidth]
{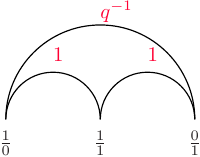}  
\end{center}
that remains undeformed.
Observe that~\eqref{WFS} is very similar to the weighted Pascal identity~\eqref{WPI}.

Here is a part of the $q$-deformed Farey graph.
\begin{center}
\includegraphics[width=0.7\textwidth]
{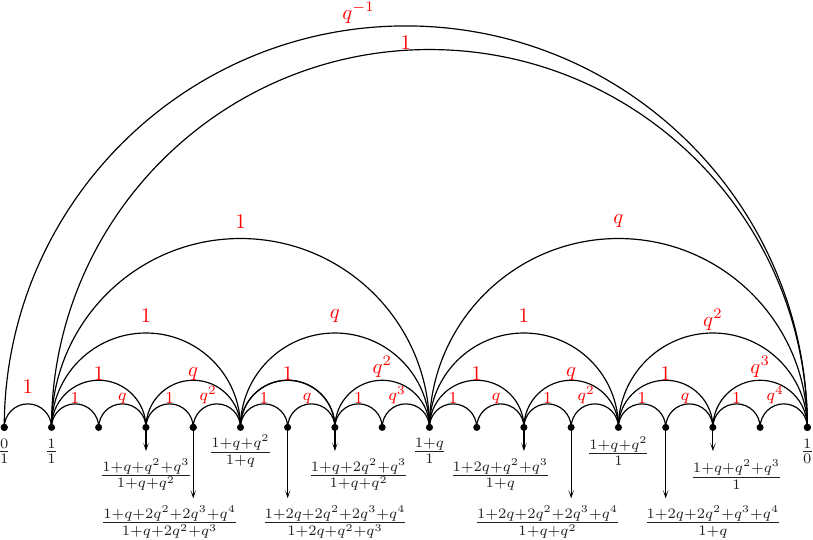}
\end{center}

\begin{thm}
\label{FereyThm}
(i)
The rational functions obtained from the $q$-deformed Farey graph are precisely
the $q$-rationals from Section~\ref{ItroSec}.

(ii)
Every $q$-rational can be calculated this way.
\end{thm}

It is proved in~\cite{MGOfmsigma} that the obtained rational functions are those
given by~\eqref{qa} and~\eqref{qc}.
The proof is based on the connection between the Farey graph and continued fractions.
The fact that every $q$-rational appears in the $q$-deformed Farey graph
is a corollary of the well-known statement that every rational can be calculated from the initial triangle
in the classical Farey graph.

\section{Properties of $q$-rationals}

We outline the basic properties of $q$-deformed rational numbers.
More information can be found in the fast growing literature.
We emphasize close resemblance to the properties of $q$-binomials.

\subsection{Total positivity}
Does the quantization preserve order?
Suppose we have two rationals~$\frac{n}{m}>\frac{n'}{m'}$, 
can we say that $\left[\frac{n}{m}\right]_q>\left[\frac{n'}{m'}\right]_q$ in some sense?
An affirmative answer is given by the next result.

\begin{thm}
\label{PosPropBis}
Given two rationals, $\frac{n}{m}>\frac{n'}{m'}$, consider their $q$-analogues
$\left[\frac{n}{m}\right]_q=\frac{\cN(q)}{\cM(q)}$ and $\left[\frac{n'}{m'}\right]_q=\frac{\cN'(q)}{\cM'(q)}$.
The polynomial 
$$
\mathcal{X}_{\frac{n}{m},\frac{n'}{m'}}:=
\cN(q)\cM'(q)-\cM(q)\cN'(q)
$$
has positive integer coefficients.
\end{thm}

This statement has a {\it topological nature}
since every ordered set is equipped with a natural topology.
It will be crucial for extension of the theory of $q$-numbers to irrationals.

This property was understood in~\cite{MGOfmsigma} as an instance of the ``total positivity'' property,
which is a fundamental mathematical concept.
Substituting $\frac{1}{m}$ or $\frac{0}{1}$ for one of the numbers $\frac{n}{m}$ or $\frac{n'}{m'}$,
Theorem~\ref{PosPropBis} implies that for every rational number
$\frac{n}{m}$ the corresponding polynomials $\cN(q)$ and $\cM(q)$
have positive integer coefficients.
The proof of~\cite{MGOfmsigma} 
reduces Theorem~\ref{PosPropBis} to this particular case.

An important corollary of Theorem~\ref{PosPropBis} is:
$$
\mathcal{X}_{\frac{n}{m},\frac{n'}{m'}}=q^\a,
\qquad
\hbox{if and only if}
\quad
nm'-mn'=1,
$$
for some integer~$\a\geq 0$.

\subsection{Unimodality and polindromicity}
Let us look at the shape of the polynomials arising as numerators and denominators of $q$-rationals.

Examples of $q$-Fibonacci and $q$-Pell numbers
considered in Section~\ref{FiboPell} give a hope for the unimodality property.
This turns out to be true and constitutes a property similar to that of $q$-binomial coefficients.

\begin{thm}
\label{UniTh}
For every $q$-rational $\left[\frac{n}{m}\right]_q=\frac{\cN(q)}{\cM(q)}$,
the sequences of coefficients of $N(q)$ and $M(q)$ are unimodal.
\end{thm}

This was conjectured in~\cite{MGOfmsigma}, but a complete proof took some time
(well, a little less than in the Sylvester case!).
After the efforts of many authors, the proof of this theorem was eventually obtained in~\cite{OgRa}.

We have seen in many examples that the polynomials $\cN(q)$ and $\cM(q)$ are not necessarily palindromic.
It turns out however, that this property holds for another notable family of polynomials.
To every $A\in\PSL(2,\Z)$ we have associated an operator $A_q$ acting on $\Z(q)$ by linear-fractional transformations;
see Section~\ref{PSLAct}.
Lifted to $\SL(2,\Z(q))$, the operator $A_q$ becomes a $2\times2$ matrix defined up to a scalar multiple.
The following result does not depend on the choice of the lift.

\begin{thm}
\label{PalTh}
For every $A\in\PSL(2,\Z)$ the trace polynomial~$\mathrm{Tr}(A_q)$ is palindromic.
\end{thm}

This statement is proved in~\cite{LMGadv}.
Let us give here an example to illustrate this statement.

\begin{exe}
\label{Ex11}
(i) Calculate the $q$-analogue of the following matrix 
$$
B=
\begin{pmatrix}
5&2\\[2pt]
2&1
\end{pmatrix},
\qquad\qquad
B_q=
\begin{pmatrix}
q+2q^2+q^3+q^4&1+q\\[2pt]
q+q^2&1
\end{pmatrix},
$$
where $B_q$ is of course defined up to a scalar multiple.

(ii) Observe that its trace is palindromic: $\mathrm{tr}(B_q)=1+q+2q^2+q^3+q^4$.
\end{exe}

The $q$-analogue of the trace function on $\PSL(2,\Z)$ seems to be an interesting notion
that deserves a further study and can potentially have many applications.
So far it was used to define $q$-deformations 
of the Markov numbers; see~\cite{LMGadv}.

\subsection{Enumerative combinatorics of $q$-rationals}

A similarity between $q$-rationals and $q$-binomials was observed in~\cite{MGOfmsigma}.
This similarity was reinforced in other works on the subject.

{\bf A}.
A combinatorial meaning of the coefficients of the polynomials $\cN(q)$ and $\cM(q)$
in a $q$-rational $\left[\frac{n}{m}\right]_q=\frac{\cN(q)}{\cM(q)}$ was suggested in~\cite{MGOfmsigma}
and beautifully reformulated in~\cite{Ove}.
The statement is exactly the same as
that in the case of $q$-binomials, but the rectangle is replaced by the so-called ``snake graph''.

Let $x=[a_1,\ldots,a_{2\ell}]$ be the standard continued fraction expansion of a rational~$x\geq1$.
The {\it snake graph} associated with~$x$ is the collection of $a_1+\cdots+a_{2\ell}-1$ boxes in the square lattice,
which the snake visites when crawling:  $a_1-1$ steps up, $a_2$ steps to the right, $a_3$ steps up, etc., 
ending with $a_{2\ell}-1$ steps to the right.
As in the case of $q$-binomials, every north-east lattice path
 is the boundary of a Young diagram.
But this time one only takes the paths with vertices in the snake graph.

\begin{thm}
\label{NEThmqR}
Given a rational $\frac{n}{m}\geq1$ and its $q$-analogue $\left[\frac{n}{m}\right]_q=\frac{\cN(q)}{\cM(q)}$,
the coefficient of $q^m$ in~$\cN(q)$ is the number of north-east lattice paths in the snake graph
with $m$ under boxes.
\end{thm}

The denominator $\cM(q)$ has a similar interpretation with a smaller snake graph.

\begin{ex}
Consider again $\frac{5}{2}=[2,2]$ and $\frac{5}{3}=[1,1,1,1]$,
the snake graphs of $\frac{5}{2}$ and $\frac{5}{3}$ are as follows
$$
\xymatrix @!0 @R=0.6cm @C=0.7cm
{
\bullet\ar@{-}[rr]\ar@{-}[dd]&\bullet\ar@{-}[dd]
&\bullet\ar@{-}[d]\\
\bullet\ar@{-}[rr]&\bullet&\bullet
\\
\bullet\ar@{-}[r]
&\bullet&
}
\qquad\qquad
\xymatrix @!0 @R=0.6cm @C=0.7cm
{
&\bullet\ar@{-}[dd]\ar@{-}[r]
&\bullet\ar@{-}[dd]\\
\bullet\ar@{-}[rr]\ar@{-}[d]&\bullet&\bullet
\\
\bullet\ar@{-}[rr]
&\bullet&\bullet
}
$$
Counting north-east lattice paths in these snake graphs,
one gets the polynomials from Exercice~\ref{LEx},~(ii).
For instance, there are exactly two paths in the snake graph of~$\frac{5}{2}$ with one under box:
$$
\xymatrix @!0 @R=0.6cm @C=0.7cm
{
\bullet\ar@{-}[r]\ar@{-}[dd]&{\color{red}\bullet}\ar@{<=}[d]\ar@{=>}[r]
&{\color{red}\bullet}\ar@{-}[d]\\
\bullet\ar@{-}[rr]&{\color{red}\bullet}\ar@{<=}[d]&\bullet
\\
{\color{red}\bullet}\ar@{=>}[r]
&{\color{red}\bullet}&
}
\qquad\qquad
\xymatrix @!0 @R=0.6cm @C=0.7cm
{
\bullet\ar@{-}[rr]\ar@{-}[d]&\bullet\ar@{-}[dd]
&{\color{red}\bullet}\ar@{<=}[d]\\
{\color{red}\bullet}\ar@{=>}[r]\ar@{<=}[d]&{\color{red}\bullet}\ar@{=>}[r]&{\color{red}\bullet}
\\
{\color{red}\bullet}\ar@{-}[r]
&\bullet&
}
$$
and this corresponds to the coefficient~$2$ in the numerator of~$\left[\frac{5}{2}\right]_{q}$.
\end{ex}

{\bf B}.
The following result belongs to discrete geometry.
It is also very similar to that of the $q$-binomials.

\begin{thm}
\label{OveThm}
Let $q$ be a power of a prime integer, then
$\cN(q)$ is the number of points in the collection of Schubert cells in the Grassmannian 
$\mathrm{Gr}_{r,s}(\mathbb{F}_q)$
with
$
r=a_1+a_3+\cdots+a_{2\ell-1}$
and
$s=a_1+a_2+\cdots+a_{2\ell},
$
that correspond to all of the Young diagrams that fit in the snake graph.
\end{thm}

It was deduced in~\cite{Ove} from Theorem~\ref{NEThmqR}.

{\bf C}.
The question of connection of $q$-rationals to quantum calculus and quantum algebra remains wide open.
First steps have been made in~\cite{Tho} where an infinitesimal version of the $\PSL(2,\Z)$-action was considered.

\section{$q$-analogues of irrational numbers}

The most important property of $q$-rationals is explained in this section.
Suppose we have a sequence of rational numbers $x_n$ converging to some {\it irrational}~$x\in\R$,
is there any way to say that the sequence of rational functions~$\left[x_n\right]_q$ converges?
A positive answer was suggested in~\cite{MGOexp}:
the sequence $q$-rationals~$\left[x_n\right]_q$ converges in the topology of formal power series.
This allows one to introduce the concept of $q$-irrational $[x]_q$
which is a formal power series in~$q$ with integer coefficients.

End of the story? Definitely not!
It turns out that, for every irrational~$x$,  the power series
$[x]_q$ is actually an analytic function in~$q$ (viewed as a complex parameter, $q\in\C$)
in some domain around~$0$.
Conjectured in~\cite{LMGOV}, this was recently proved in~\cite{Eti}.

The subject is an active area of research that contain many open problems.

\subsection{Stabilization phenomenon}

Recall that the space of power series 
$X(q)=\sum_{k\geq0}X_{k}\,q^k$
denoted by $R[[q]]$, where $R$ is an arbitrary ring (in our case $R=\Z$),
has a standard and very simple {\it topology}.
A sequence of power series 
$$
X_n(q)=\sum_{k\geq0}X_{n,k}\,q^k
$$
is said to {\it converge} to some power series $X(q)$ if the coefficients~$X_{n,k}$ 
stabilize, as~$n$ grows.

\begin{rem}
This formal convergence of sequences of power series should not be confounded
with the question of convergence of power series in complex variable!
\end{rem}

Let $x>1$ be an irrational number, and let $(x_n)_{n\in\N}$
be a sequence of rationals converging to~$x$.
We also assume that $x_n\geq1$ which implies that 
the rational functions $\left[x_n\right]_q$ have no singularity at~$q=0$.
Consider the Taylor expansions of $\left[x_n\right]_q$ at~$q=0$,
that by abusing of notation, will be dented by the same symbol
$
\left[x_n\right]_q=\sum_{k\geq0}\varkappa_{n,k}\,q^k.
$
So far, we deal with these series as formal power series in~$q$.

\begin{thm}
\label{StabThm}
For every irrational~$x>1$ and any rational approximation $x_n\to{}x$, the Taylor series 
of the functions~$\left[x_n\right]_q$ converges to some power series from~$\Z[[q]]$,
which depends only on~$x$.
\end{thm}

\noindent
The result of this ``stabilization process'' is called the $q$-analogue of~$x$ and is denoted by~$[x]_q$.
Theorem~\ref{StabThm} extends to irrationals $x<1$ using the operator $T_q^{-1}$; see~\eqref{AcTq},
but in this case $[x]_q$ becomes a Laurent series in~$q$ (containing terms with~$q^{-1}$).

Theorem~\ref{StabThm} was first obtained by computer experimentation and looked as a miracle.
A simple proof was then obtained in~\cite{MGOexp} based on the total positivity Theorem~\ref{PosPropBis}.
This may come as a surprise at first sight, but both statements are topological:
the topology of ordered set on~$\Q$ vs the topology of power series.

The most simple way to calculate the $q$-analogue of an irrational $x\in\R\setminus\Q$ is to consider
the infinite continued fraction expansion
$x=[a_0,a_1,a_2,a_3,\ldots]$.
The explicit formula for the series $[x]_{q}$ is then
\begin{equation}
\label{QuantIrrat}
[x]_{q}\quad=\quad
[a_0]_{q} + \cfrac{q^{a_{0}}}{[a_1]_{q^{-1}} 
          + \cfrac{q^{-a_{1}}}{[a_{2}]_{q} 
          +\cfrac{q^{a_{2}}}{[a_{3}]_{q^{-1}}
          + \cfrac{q^{-a_{3}}}{
    \ddots}
          } }} 
\end{equation}
which is nothing but the infinite version of~\eqref{qa}.
The stabilization phenomenon guarantees that cutting $n$ terms of
this infinite continued fraction, at least $n$ terms of its Taylor series will be stabilized.

\subsection{Quadratic irrationals: gold, silver, bronze, and other precious metals}\label{GoldSec}
To illustrate Theorem~\ref{StabThm} and get simple examples,
let us consider a remarkable series of quadratic irrationals with 1-periodic continued fractions
$$
y_k=\frac{k+\sqrt{k^2+4}}{2}=[k,k,k,\ldots],
$$ 
for $k\in\N$.
The first examples are the golden ratio
$\varphi=\frac{1+\sqrt{5}}{2}$ (see Lecture~1), the silver ratio $\sqrt{2}+1$, 
followed by other irrational numbers of this series, often called the ``metallic numbers''.

The sequence of Fibonacci quotients $\frac{F_{n+1}}{F_n}$ considered in Section~\ref{FiboPell}
converges to~$\varphi$.
The corresponding Taylor series are
$$
\begin{array}{rcl}
\left[\frac{8}{5}\right]_q&=&
1 + q^2 - q^3 + 2 q^4 - 4 q^5 + 7 q^6 - 12 q^7 + 21 q^8 - 37q^9+ 65q^{10} - 114q^{11}+ 200q^{12}\cdots\\[6pt]
\left[\frac{13}{8}\right]_q&=&
1 + q^2 - q^3 + 2 q^4 - 3 q^5 + 3 q^6 - 3 q^7 + 4 q^8 - 5q^9+ 5q^{10} - 5q^{11} + 6q^{12}\cdots\\[6pt]
\left[\frac{21}{13}\right]_q&=&
1 + q^2 - q^3 + 2 q^4 - 4 q^5 + 8 q^6 - 17 q^7 + 36 q^8 - 75q^9+ 156q^{10} - 325q^{11} + 677q^{12}\cdots\\[6pt]
\cdots&\cdots&\cdots
\end{array}
$$
More and more terms remain unchanged as $n$ grows.
The series eventually stabilize as $n$ grows and the result of this stabilization is
\begin{equation}
\label{GRSer}
\begin{array}{rcl}
\left[\varphi\right]_q&:=&
1 + q^2 - q^3 + 2 q^4 - 4 q^5 + 8 q^6 - 17 q^7 + 37 q^8 - 82 q^9 + 185 q^{10} \\[6pt]
&&- 423 q^{11} + 978 q^{12}-2283q^{13}+ 5373q^{14}-12735q^{15}+30372q^{16} \cdots
\end{array}
\end{equation}
Quite remarkable, the sequence of coefficients of this series coincides (up to an alternating sign and the first term)
with the sequence A004148 of OEIS called the ``generalized Catalan numbers''.

Similarly, the quotients of Pell's numbers $\frac{P_{n+1}}{P_n}$ converges to $\sqrt{2}+1=[2,2,2,\ldots]$
which allows us to calculate the series $\left[\sqrt{2}+1\right]_q$.

\begin{exe}
\label{GoldEx}
(i)
Use the periodic continued fraction \eqref{QuantIrrat} to calculate the generating functions:
$$
\begin{array}{rcl}
\left[\varphi\right]_q&=&
\displaystyle
\frac{q^2+q-1+\sqrt{(q^2 + 3q + 1)(q^2 - q + 1)}}{2q},
\\[12pt]
\left[\sqrt{2}+1\right]_q&=&
\displaystyle
\frac{q^3+2q-1+\sqrt{(q^4 + q^3 + 4q^2 + q + 1)(q^2 - q + 1)}}{2q}.
\end{array}
$$

(ii) (see~\cite{OP}).
The general expression for the metallic $q$-numbers is
$$
\left[y_k\right]_q=
\frac{1}{2q}\left(q[k]_q+(q^k+1)(q-1)+\sqrt{\left(q[k]_q+(q^k+1)(q-1)\right)^2+4q}\right).
$$
\end{exe}

A general theory of $q$-deformed quadratic irrationals was developed in~\cite{LMGadv}.
It was proved that if $x$ is a quadratic irrational, i.e. a root of a quadratic equation
$
ax^2+bx+c=0
$
with $a,b,c\in\Z$, then $\left[x\right]_q$
satisfis an equation of the form
$$
A(q)\left[x\right]_q^2+B(q)\left[x\right]_q+C(q)=0,
$$
where $A(q),B(q),C(q)$ are some polynomials in~$q$.
A more precise statement is the following.

\begin{thm}
\label{QuadThm}
For every quadratic irrational $x$, the $q$-analogue $\left[x\right]_q$ can be written in the form
$$
\left[x\right]_q=\frac{\Pc(q)+\sqrt{\Qc(q)}}{\cR(q)},
$$
where $\Pc,\Qc,\cR$ are polynomials in~$q$.
Moreover, the polynomial $\Qc$ is palindromic.
\end{thm}

We could observe this palindromicity in the examples of metallic numbers and it turns out to be a universal property.
The proof follows from the fact that every quadratic irrational is a fixed point of some element of~$\PSL(2,\Z)$,
and the polynomial~$\Qc(q)$ is closely related to its trace.

\subsection{Cubic irrationals,
the quantum regular heptagon and nonagon}\label{CubSec}
Suppose now that $x$ is a cubic irrational, that is a solution to
an irreducible cubic equation
\begin{equation}
\label{Cubic}
ax^3+bx^2+cx+d=0
\end{equation}
with $a,b,c,d\in\Z$.
Does $\left[x\right]_q$ satisfy any $q$-analogue of~\eqref{Cubic} with polynomial in~$q$ coefficients?
The answer to this question is unknown, but it was conjectured in~\cite{OU} that this is never true.
The conjecture is based on computer experimentation.
The reader may remember (see Lecture~4) explicit formulas for solutions of a cubic equation.
However, there is no hope for any $q$-version of such formula in the cubic case.

It turns out however, that some ``miracle'' happens for special cubic equationss.
Let us consider two remarkable examples:
\begin{equation}
\label{HeptaEq}
x^3 + x^2 -2x - 1 = 0,
\qquad\hbox{and}\qquad
x^3 -3x + 1 = 0,
\end{equation}
related to the regular heptagon and nonagon, respectively.
Their trigonometric solutions are known since several centuries:
$$
\begin{array}{rcl}
x_1=2\cos\left(\frac{2\pi}{7 }\right),
&
x_2=2\cos\left(\frac{4\pi}{7 }\right),
&
x_3=2\cos\left(\frac{8\pi}{7 }\right);\\[4pt]
x_1=2\cos\left(\frac{2\pi}{9}\right),
&
x_2=2\cos\left(\frac{4\pi}{9}\right),
&
x_3=2\cos\left(\frac{8\pi}{9}\right),
\end{array}
$$ 
for the first and the second equation~\eqref{HeptaEq}, respectively.

\begin{thm}
\label{CubThm}
The $q$-deformations of the roots of~\eqref{HeptaEq} satisfy the equation
\begin{equation}
\label{CubicQuant}
\begin{array}{rcl}
X^3 - b(q)X^2 - \left(q^{-1}b(q)+3q^{-2}\right)X - q^{-3} &=& 0,\\[4pt]
X^3 - \tilde{b}(q)X^2 + \left(\tilde{b}(q)-3\right)X + 1 &= &0,
\end{array}
\end{equation}
respectively.
The coefficients~$b(q)$ and~$\tilde{b}(q)$  in~\eqref{CubicQuant}
are a certain Laurent series in~$q$.
\end{thm}

The proof of Theorem~\ref{CubThm} is entirely based on the $\PSL(2,\Z)$-invariance
and we will give some details.
But let us first reformulate Theorem~\ref{CubThm} using Vieta's formulas.
Recall that if $x_1,x_2,x_3$ are roots of an equation $x^3+bx^2+cx+d=0$, then
$$
\begin{array}{rcl}
b & = & -x_1-x_2-x_3,\\
c & = & x_1x_2+x_1x_3+x_2x_3,\\
d & = & -x_1x_2x_3;
\end{array}
$$
see again Lecture~4.

In our case, let $X_i=\left[x_i\right]_q$, where $i=1,2,3$, be the $q$-deformations of the roots.
No $q$-analogue of any of the three Vieta's formulas for $X_1,X_2,X_3$ holds for general cubic equation.
But, for our two equations~\eqref{HeptaEq}, we have
\begin{equation}
\label{Vieta1}
\begin{array}{rcl}
X_1X_2X_3&=&q^{-3},\\[4pt]
X_1X_2+X_2X_3+X_3X_1&=&
q^{-1}\left(X_1+X_2+X_3\right)-3q^{-2}.
\end{array}
\end{equation}
and
\begin{equation}
\label{Vieta2}
\begin{array}{rcl}
X_1X_2X_3&=&-1,\\[4pt]
X_1X_2+X_2X_3+X_3X_1&=&X_1+X_2+X_3-3.
\end{array}
\end{equation}
respectively.
Remember that $X_1,X_2,X_3$ are infinite series in $q$, the above equalities are result of spectacular cancellation of terms.

\begin{exe}
Show that \eqref{Vieta1} and~\eqref{Vieta1} is equivalent to the respective equation~\eqref{CubicQuant}.
\end{exe}

Let us go back to the proof of Theorem~\ref{CubThm}.

\begin{exe}
Check that equations~\eqref{HeptaEq} are invariant under the linear-fractional transformations
$$
x\mapsto-\frac{x+1}{x}
\qquad\hbox{and}\qquad
x\mapsto\frac{x-1}{x},
$$
respectively.
\end{exe}

These transformations relate the three roots of~\eqref{HeptaEq} to each other.
For instance, for the heptagon equation we have
$$
x_1=-\frac{x_2+1}{x_2},
\qquad
x_2=-\frac{x_3+1}{x_3},
\qquad
x_3=-\frac{x_1+1}{x_1}.
$$
In more sophisticated terms, the Galois group of the equations~\eqref{HeptaEq}, which is the cyclic group of order~3,
admits an embedding into~$\PSL(2,\Z)$.

The $q$-deformation then goes along the following lines.
The above linear-fractional transformations
correspond to the action of the elements of~$\PSL(2,\Z)$
$$
T^{-1}S=
\begin{pmatrix}
1&1\\[4pt]
-1&0
\end{pmatrix},
\qquad\hbox{and}\qquad
TS=
\begin{pmatrix}
1&-1\\[4pt]
1&0
\end{pmatrix}
$$
which are its only (up to conjugation) third-order elements.
After the $q$-deformation defined by~\eqref{AcTq}, these elements become
$$
T_q^{-1}S_q=
\begin{pmatrix}
q&1\\[4pt]
-q^2&0
\end{pmatrix}.
\qquad\hbox{and}\qquad
T_qS_q=
\begin{pmatrix}
1&-1\\[4pt]
1&0
\end{pmatrix}
$$
Note that the element $T_qS_q\in\PSL(2,\Z)$ does not depend on~$q$.
It will appear again in Section~\ref{CataS}.

Let us deduces~\eqref{Vieta1} and~\eqref{Vieta2} from the invariance of the $q$-deformed equations under the action
of these elements.
(For simplicity, we consider only the second case, the first is similar.)
By $T_qS_q$-invariance, we have
$$
X_2(q)=-\frac{1}{X_1(q)-1}.
$$
Therefore,
$$
X_1(q)X_2(q)X_3(q)=-X_1(q)\frac{1}{X_1(q)-1}\frac{X_1(q)-1}{X_1(q)}=-1,
$$
hence the first formula~\eqref{Vieta2}.
Consider now the product $\left(X_1(q)-1\right)\left(X_2(q)-1\right)\left(X_3(q)-1\right)$.
Applying
$$
X_1(q)=\frac{X_2(q)-1}{X_2(q)},
\qquad
X_2(q)=\frac{X_3(q)-1}{X_3(q)},
\qquad
X_3(q)=\frac{X_1(q)-1}{X_1(q)},,
$$ 
we have
$$
\left(X_1(q)-1\right)\left(X_2(q)-1\right)\left(X_3(q)-1\right)=
\left(X_1(q)X_2(q)X_3(q)\right)^2=1,
$$
which readily gives the second formula~\eqref{Vieta2}.

This completes the proof of Theorem~\ref{CubThm}.

\begin{rem}
One can say that ``two out of three'' Vieta's formulas, \eqref{Vieta1} and~\eqref{Vieta2} 
are still satisfied for the $q$-deformed roots of~\eqref{HeptaEq}, but not all three of them!
The coefficients $b(q)$ and~$\tilde{b}(q)$ in~\eqref{CubicQuant} can be calculated as the sum of the $q$-deformed roots: 
$$
b(q),\tilde{b}(q) = X_1+X_2+X_3.
$$
Properties of these coefficients are unknown.
It would be interesting to understand the nature of these $q$-series and calculate their expressions 
in terms of hypergeometric series.

Let us also mention that equations~\eqref{HeptaEq} are included into two families for which Theorem~\ref{CubThm} still holds,
but we do not dwell of the details heres.
\end{rem}

\subsection{The radius of convergence: a strong conjecture and a weaker result}
So far, the $q$-analogue $\left[x\right]_q$ of an irrational $x\in\R$ is understood as a formal series in~$q$.
Let us think of~$q$ as a complex variable,
it is a very natural question whether this series has some radius of convergence~$R(x)$, so that $\left[x\right]_q$
is a  holomorphic function in~$q$ in some disc (or more generally some domain) of~$\C$.

Consider first the example of the golden ratio $\varphi=\frac{1+\sqrt{5}}{2}$.
The radius of convergence $R(\varphi)$ of the series~\eqref{GRSer} representing
the $q$-deformation $\left[\varphi\right]_q$ 
can be calculated explicitly.
The generating function of~$\left[\varphi\right]_q$ was calculated in Example~\ref{GoldEx},
it follows that $R(\varphi)$ is equal to the modulus of the minimal root of
the polynomial $q^2 + 3q + 1$ (under the radical).

\begin{exe}
Check that the smallest root of $q^2 + 3q + 1$ has modulus $\frac{3-\sqrt{5}}{2}$, and therefore
$$
R(\varphi)=\frac{3-\sqrt{5}}{2}=0.381966...
$$
\end{exe}

The following conjecture was formulated in~\cite{LMGOV}.

\begin{conj}
\label{MainConj}
 For every real $x>0$ the radius of convergence $R(x)$ of the series~$\left[x\right]_q$
satisfies the inequality
$
R(x)\geq R(\varphi)
$
and the equality holding only for $x$ which are $\PSL(2,\Z)$-equivalent to~$\varphi$.
\end{conj}

This statement, provided it is true, is a $q$-analogue of the classical Hurwitz theorem
that $\varphi$ is ``the most irrational number'' (see Lecture~1).
It was checked by a long sequence of computer experimentation and proved for some special cases.
For rational~$x$, this conjecture was proved in~\cite{EGMS}
stating the for every $q$-rational $\left[\frac{n}{m}\right]_q=\frac{\cN(q)}{\cM(q)}$
the roots of the polynomial $\cM(q)$ cannot belong to the disc with radius $R(\varphi)$.

The known result recently proved~\cite{Eti} for irrational~$x$ is twice weaker.

\begin{thm}[\cite{Eti,LMGOV}]
\label{NewThm}
For every $x>0$, the radius of convergence $R(x)$ of the series~$\left[x\right]_q$
has the following lower bound
$$
R(x) > 3-2\sqrt{2}=0.171572...
$$
\end{thm}

This statement was proved in~\cite{LMGOV} only for rational~$x$,
the proof is based on application of Rouch\'e theorem of complex analysis.
The proof of~\cite{Eti} is different and based on the analysis of $q$-adic convergence of the continued fraction~\eqref{QuantIrrat}.
Surprisingly, two different approaches give the same bound $3-2\sqrt{2}$,
which is however far from the expected bound $\frac{3-\sqrt{5}}{2}$.

Let us give one more example of a transcendental number.
The $q$-analogue of $\pi$ is given by series that starts as follows.
$$
\begin{array}{rcl}
\left[\pi\right]_q&=&
1+q+q^2+q^{10}-q^{12}-q^{13}+q^{15}+q^{16} - q^{20}-2q^{21}- q^{22}+2q^{23}+4q^{24}+q^{25}\\[4pt]
&&-4q^{27}-4q^{28}-2q^{29}+q^{30}+5q^{31}+8q^{32}+3q^{33}-3q^{34}-10q^{35}-12q^{36}-5q^{37}\\[4pt]
&&+8q^{38}+19q^{39}+20q^{40}+2q^{41}-18q^{42}-32q^{43}-25q^{44}+31q^{46}+51q^{47}+45q^{48}-7q^{49} \cdots
\end{array}
$$
The coefficients of this series grow slowly, further computer estimations suggest that its radius of
convergence equals~$1$, so that $\left[\pi\right]_q$ is analytic in the unit circle.
However, this is also a conjecture.

\section{Surprises}

In this final part of the lecture we collect some surprising properties of $q$-numbers.
Some of them are very unexpected.

\subsection{Left $q$-rationals}\label{FlatSec}
What can we say about the stabilization phenomenon in this rational case?
Suppose that we have a sequence of rationals $(x_n)_{n\in\N}$ such that $x_n\to{}x$,
and $x$ itself is rational.
Is it true that the sequence of rational functions $\left[x_n\right]_q$ has some limit and if this is the case,
does this limit coincide with $\left[x\right]_q$?
The answer is ``yes'' when the sequence $(x_n)$ approximates~$x$ from the right 
and ``no'' if it approximates from the left.
This phenomenon was observed in~\cite{MGOexp} and developed in~\cite{BBL},
it gives rise of an alternative ``left'' version of $q$-rationals.

Recall the definition of $q$-rationals from Section~\ref{ItroSec}.
If $x\in\Q$, then there exist (infinitely many) elements $A\in\PSL(2,\Z)$ such that $x=A(0)$.
Take the $q$-deformed action~$A_q$ on rational functions (Section~\ref{PSLAct}), 
the $q$-analogue of~$x$ is then defined by
$\left[x\right]_q:=A_q(0)$.
The result is independent of the choice of $A$.
Set
\begin{equation}
\label{Flat}
\left[x\right]_q^\flat:=A_q\left(\frac{q-1}{q}\right).
\end{equation}
The definition is correct, i.e. the result of~\eqref{Flat} is independent of the choice of $A$.
This follows from the next exercise.

\begin{exe}
Check that $\left[0\right]_q=0$ and $\left[0\right]_q^\flat=\frac{q-1}{q}$ have the same stabiliser:
the infinite cyclic group generated by~$L_q$.
\end{exe}

We have the following.

\begin{thm}
\label{FlatThm}
(i) 
If a sequence of rationals $(x_n)$ approximates a rational~$x$ from the right, then
the Taylor series of~$\left[x_n\right]_q$ stabilize to that of~$\left[x\right]_q$.

(ii) 
If the sequence of rationals $(x_n)$ approximates a rational~$x$ from the left, then
the Taylor series of~$\left[x_n\right]_q$ stabilize to that of~$\left[x\right]_q^\flat$.
\end{thm}

To prove this, it suffices to check the statement for $x=0$ and then use the $\PSL(2,\Z)$-invariance.

\begin{exe}
Calculate examples of left $q$-rationals:
$\left[1\right]_q^\flat=q,\;
\left[2\right]_q^\flat=q^2+1,\;
\left[\infty\right]_q^\flat=\frac{1}{1-q}$.
\end{exe}

We can say now that the $q$-deformed real line has the following structure.
Every irrational is represented by the unique $q$-analogue, while rationals are ``doubled''.
The map $x\mapsto\left[x\right]_q$ is continuous at irrational points and discontinuous
if $x$ is rational.
In other words, a sequence of rationals $(x_n)$ approximates a rational~$x$
neither from the left, nor from the right, then 
the sequence $\left[x_n\right]_q$ has no limit in any sense.

\begin{rem}
In order to better understand the nature of left $q$-rationals, recall that an element of $\SL(2,\Z)$ 
can be of one of three types:
\begin{enumerate}
\item
$A\in\SL(2,\Z)$ is hyperbolic if $|\mathrm{tr}(A)|\geq3$;
\item
$A\in\SL(2,\Z)$ is parabolic if $|\mathrm{tr}(A)|=2$;
\item
$A\in\SL(2,\Z)$ is elliptic if $|\mathrm{tr}(A)|=0$, or $1$.
\end{enumerate}

Every hyperbolic element has two distinct fixed points and these are quadratic irrationals.
A parabolic element $A\in\SL(2,\Z)$ has one double fixed point, left and right $q$-rationals
are two solutions to the equation $A_q(X)=X$, where $A_q$ is the $q$-deformed action of~$A$, 
and $X$ is a rational function in~$q$.

Elliptic elements have complex fixed points.
They also play some role in the picture, see Section~\ref{CataS} below.

\end{rem}

\subsection{Hidden symmetries}
The theory of $q$-numbers is determined by the modular group $\PSL(2,\Z)$,
but it turns out that the group of symmetry is actually bigger!
The biggest currently known group of symmetry is $\PGL(2,\Z)\times\Z_2$
(we use familiar to physicists notation $\Z_2$ for the cyclic group of two elements: $\Z_2=\Z/(2\Z)$).

The group $\PGL(2,\Z)$ is obtained by adding to $\PSL(2,\Z)$ one extra generator.
One can use equivalently one of the transformations:
$$
x\mapsto-x,
\qquad\hbox{or}\qquad
x\mapsto\frac{1}{x}.
$$
There exists an action of $\PGL(2,\Z)$ on the space $\Z(q)$ of rational functions that commutes with
the map $[\,.\,]:\Q\cup\{\infty\}\to\Z(q)$ defined in Section~\ref{ItroSec}.
Besides the generators $T_q$ and $S_q$ of the $\PSL(2,\Z)$-action (see~\eqref{AcTq}) we have
$$
X(q)\mapsto
\frac{-X(q)+1-q^{-1}}{\left(q-1\right)X(q)+1}
\qquad\hbox{and}\qquad
X(q)\mapsto
\frac{\left(q-1\right)X(q)+1}{qX(q)+1-q}.
$$
These maps were found in~\cite{Jou}.

Consider the following involution acting on rational functions:
$$
\mathcal{I}:
X(q)\mapsto
\frac{-X(q^{-1})+q-1}{\left(1-q\right)X(q^{-1})+q}.
$$
It commutes with the map $[\,.\,]:\Q\cup\{\infty\}\to\Z(q)$ and also with the $\PGL(2,\Z)$-action.
As proved in~\cite{Tho}, this map interchanges left and right $q$-rationals:
$$
\mathcal{I}:\left[x\right]_q\leftrightarrow\left[x\right]_q^\flat.
$$

Are there other hidden symmetries?
No-one knows...

\subsection{Hankel determinants and remarkable recurrences}\label{CataS}
Suppose you have an integer sequence $(a_n)_{n\geq1}$ and you don't know what to do with it.
A good idea can be to look at the corresponding sequence of Hankel determinants
$$
\D_0:=1,
\quad
\D_1=a_1,
\quad
\D_2=\left|
\begin{array}{rl}a_1&a_2\\
a_2&a_3
\end{array}
\right|,
\quad
\D_3=\left|
\begin{array}{rcl}a_1&a_2&a_3\\
a_2&a_3&a_4\\
a_3&a_4&a_5
\end{array}
\right|,
\ldots
$$
and see if this sequence has some interesting properties.
This has been done in~\cite{OP} for the sequence of coefficients of the $q$-deformed golden ratio~$\left[\varphi\right]_q$;
see Exercice~\ref{GoldEx}.
The corresponding sequence of Hankel determinants turns out to be $4$-antiperiodic and consist in $-1,0,1$:
$$
\D_n(\varphi)=
1,1,1,0,-1,-1,-1,0,\ldots
$$

Let us also consider the sequences of ``shifted'' Hankel determinants:
$$
\D_0^{(k)}:=1,
\quad
\D_1^{(k)}=a_{k+1},
\quad
\D_2^{(k)}=\left|
\begin{array}{rl}a_{k+1}&a_{k+2}\\
a_{k+2}&a_{k+3}
\end{array}
\right|,
\quad
\D_3^{(k)}=\left|
\begin{array}{rcl}a_{k+1}&a_{k+2}&a_{k+3}\\
a_{k+2}&a_{k+3}&a_{k+4}\\
a_{k+3}&a_{k+4}&a_{k+5}
\end{array}
\right|,
\ldots
$$
for $k=1,2,3,\ldots$
It turns out that the above described property of the coefficients of~$\left[\varphi\right]_q$ 
holds for three more shifted sequences.
They are again $4$-antiperiodic and consist in $-1,0,1$:
$$
\begin{array}{rclrrrrrrrr}
\D^{(1)}_n(\varphi)&=& 1,&\,0,&-1,&\;1,&-1,&\,0,&\,1,&-1,&\ldots\\[4pt]
\D^{(2)}_n(\varphi)&=& 1,&\,1,&\,1,&\,0,&-1,&-1,&-1,&\,0,&\ldots\\[4pt]
\D^{(3)}_n(\varphi)&=& 1,&-1,&\,0,&\,0,&-1,&\,1,&\,0,&\,0,&\ldots
\end{array}
$$

A similar phenomenon is well known in combinatorics.
It occurs for two celebrated integer sequences called the \textit{Catalan numbers} and the \textit{Motzkin numbers}:
$$
C_n=1, 1, 2, 5, 14, 42,\ldots
\qquad
M_n=1, 1, 2, 4, 9, 21, 51,\ldots
$$ 
Both sequences have hundreds of combinatorial interpretations and applications.
They appear in many seemingly unrelated problems connecting them, sometimes unexpectedly.

Let us briefly recall the definition and basic properties of the Catalan and Motzkin sequences.
The Catalan numbers $C_n= \frac{1}{n+1}\binom{2n}{n}$ count the number of triangulations of the convex $(n+2)$-gon,
they were known to Euler who calculated (but failed to prove!) 
their generating function 
$$
C(q)=\sum_{i=0}^\infty{}C_nq^n=\frac{1-\sqrt{1-4q}}{2q}.
$$
Motzkin numbers have the generating function 
$$
M(q)=\frac{1-q-\sqrt{(1+q)(1-3q)}}{2q^2}=\frac{1}{q}C\Big(\frac{q}{1+q}\Big),
$$ 
and this can be taken as their definition.
Among multiple combinatorial definitions, we mention that $M_n$ is the number of paths in $\Z^2$-grid
from $(0,0)$ to $(0,n)$ with steps $(1,1), (1,0), (1,-1)$.

It has been known for a long time that the Catalan numbers is the only sequence for which
two first Hankel sequences, $\D_n(C)$ and $\D^{(1)}_n(C)$, are identically~$1$.
The Motzkin numbers have the following sequences of Hankel determinants
$$
\D_n\left(M\right)=1,\,1,\,1,\ldots
\qquad\qquad
\D^{(1)}_n\left(M\right)=1,\;1,\,0,-1,-1,\,0,\ldots
$$
The first is again identically $1$, while the second is $3$-antiperiodic.
The Motzkin numbers are characterized by this property; see~\cite{Aig}.

All $q$-deformed metallic numbers~$y_k$ (see Section~\ref{GoldSec}) have similar property.
The greater is~$k$, the more shifted Hankel sequences are (anti)periodic and consist in $-1,0,1$ only.
Conjectured in~\cite{OP}, this was proved in~\cite{HP}, but the proof is very difficult!

But the surprises don't stop there and here is ``the icing on the cake''.

\begin{exe}
Check that
the first three sequences of Hankel determinants
$\D_n(\varphi),\D^{(1)}_n(\varphi)$, and $\D^{(2)}_n(\varphi)$ 
of the $q$-deformed golden ratio satisfy the
recurrence
\begin{equation}
\label{Somos4}
d_{n+4}d_n =d_{n+3}d_{n+1} - d_{n+2}^2.
\end{equation}
\end{exe}

This simple quadratic recurrence is (a version of) the celebrated Somos-4 recurrence.
Invented in the middle of 1980's it became very popular because of the relation to elliptic functions and cluster algebra.
It is also known to generate integrable dynamical system.
Other $q$-deformed metallic numbers~$\left[y_k\right]_q$ also have a similar property.
Is this relation to discrete integrable systems a simple coincidence?
We think not, but this is still a mystery.

\subsection{Trying to connect $q$-numbers to classical sequences}

We finish the lecture by a naive question.
{\it Are the generating functions of the Catalan and Motzkin numbers $C(q)$ and $M(q)$ themselves $q$-numbers?}
Indeed, their properties are close to those of $q$-metallic numbers.
A naive answer is ``no''. 
If they would, they have to be quadratic irrationals, but there is no element $A$ in $\PSL(2,Z)$ 
such that $C(q)$ or $M(q)$ is a fixed point of the $q$-deformed action $A_q$.
However, we have the following.

\begin{exe}
\label{CatMot}
Check that
\begin{equation}
\label{CatMotEq}
T_qS_q\left(C(q)\right)=qC(q)
\qquad\hbox{and}\qquad
S_q\left(M(q)\right)=qM(q)+\frac{q-1}{q}.
\end{equation}
\end{exe}

We find this quite remarkable and deserving further understanding.
Remember that, in order to recover the initial object from its
$q$-analogue, we should substitute $q=1$. 
Heuristically, for the Catalan and Motzkin generating functions we have:
$$
C(1)=\frac{1+i\sqrt{3}}{2}
\qquad\hbox{and}\qquad
M(1)=i.
$$
These are complex numbers, and
the theory of $q$-deformed complex numbers is yet to be developed! 
But for these simple examples
we are lucky, they are fixed points of $TS$ and $S$.
The functions $C(q)$ and $M(q)$ are not fixed points of $T_qS_q$ and $S_q$ in a strict sense,
but they {\it are} in some more general sense that we would be happy to understand!

\section{Solutions and hints}

We solve, or outline solutions of those exercices that do not consist in simple direct computations,
like, for instance, in the case of Exercice~\ref{PascalEx}.

\subsection*{Exercice~\ref{TransEx}}
Part (i) is a simple computation with $2\times2$ matrices.
To prove Part (ii), it suffices to show that every $x\in\Q\cup\{\infty\}$ can be sent to~$0$.
Let $x=\frac{n}{m}$ and $n>m$ (the case $n<m$ is similar), 
By induction on $\max(n,m)$, we assume that there exists $A\in\PSL(2,\Z)$ such that
$A(0)=\frac{n-m}{m}$.
Then $TA(0)=\frac{n}{m}$.

\subsection*{Exercise~\ref{Ex9}}
This statement is analogous to Lemma~\ref{FirstLm}.

\subsection*{Exercise~\ref{LEx}}
Part (i).
The first equality $\left[0\right]_q=0$ is by definition.
Since $1=T(0)$, we have 
$$
\left[1\right]_q=T_q(\left[0\right]_q)=q\left[0\right]_q+1=1.
$$
Similarly, $\frac{1}{0}=S(\frac{0}{1})$ and thus we have
$$
S_q:\left[0\right]_q\mapsto
\left[\frac{1}{0}\right]_q=-\frac{1}{q\cdot0}=\frac{1}{0}.
$$

Part (ii).
Since the continued fraction of $\frac{1}{2}=[0,2]$, we have from~\eqref{qa}
$$
\left[\frac{1}{2}\right]_q=
0+\frac{1}{1+q^{-1}}=\frac{q}{1+q}.
$$
For $\frac{5}{2}=2+\frac{1}{2}$, we have from~\eqref{qa}
$$
\left[\frac{5}{2}\right]_q=1+q+\frac{q^2}{1+q^{-1}}=\frac{1+2q+q^2+q^3}{1+q^{-1}},
$$
and similarly for $\frac{5}{3}$.

\subsection*{Exercise~\ref{Ex11}}
The matrix $B$ can be expressed in terms of the generators as follows
$B=T^3ST^2ST^2STSTS$.
Replace $T$ and $S$ by $T_q$ and $S_q$, as in~\eqref{AcTq}, to obtain~$B_q$.

\subsection*{Exercise~\ref{GoldEx}}
Since $\varphi=[1,1,1,\ldots]$ the periodic continued fraction \eqref{QuantIrrat} reads in this case
$$
\left[\varphi\right]_q
 \quad =\quad
1+ \cfrac{q}{1
          + \cfrac{q^{-1}}{ \left[\varphi\right]_q }} .
$$
It is then easy to deduce the equation 
$$
q\left[\varphi\right]^2_q-
\left(q^2+q-1 \right)\left[\varphi\right]_q -1 =0,
$$
whose formal solution is precisely the first formula of Part~(i) of the exercice.
The other cases are similar to this one. 

\subsection*{Exercise~\ref{CatMot}}
The classical recurrence relation for the Catalan numbers
$$
C_{n}=\sum _{i=1}^{n}C_{i-1}C_{n-i}
$$
(so desired but unproved by Euler!)
reads in terms of the generating function
$C(q)=1+q\,C (q)^2.$
This is equivalent to
$$
\frac{C(q)-1}{C(q)}=q\,C(q),
$$
which is precisely the first equation~\eqref{CatMotEq}.
Similarly, the second equation~\eqref{CatMotEq} is equivalent to the relation
$M(q) = 1 + qM(q) + q^2M(q)^2$.


\end{document}